\documentclass[12pt]{article}
\usepackage{amsmath,amssymb,amsthm}
\usepackage{tikz}
\usepackage{scalerel} 
\usepackage{graphicx,wrapfig,caption,subcaption,color}
 \usepackage{soul}  
 \usepackage[normalem]{ulem} 
 \usepackage[makeroom]{cancel} 
\usepackage{breqn}
\usepackage[colorlinks=true,citecolor=black,linkcolor=black,urlcolor=blue]{hyperref}

\def\nfrac#1#2{{\textstyle\frac{#1}{#2}}}
\def\dfrac#1#2{\lower0.15ex\hbox{\large$\frac{#1}{#2}$}}

\newtheorem{theorem}{Theorem}[section]
\newtheorem{lemma}[theorem]{Lemma}

\newtheorem{corollary}[theorem]{Corollary}
\numberwithin{equation}{section}

\def\kvec{{\boldsymbol{k}}}

\def\xvec{{\boldsymbol{x}}}

\def\X{{\boldsymbol{X}}}

\def\kmax{k_{\mathrm{max}}}

\def\Hrk{\mathcal{H}_r(\kvec)}

\def\E{\mathbb{E}}
\def\P{\mathbb{P}}

\def\Mk{M(\kvec)}
\def\Mkx{M(\kvec-\xvec)}
\def\Mtkx{M_2(\kvec-\xvec)}
\def\Mtk{M_2(\kvec)}
\def\purple#1{{#1}}  
   
\title{The average number of spanning hypertrees\\ in sparse uniform hypergraphs\thanks{Supported by the Australian Research Council grant DP190100977.}}

\author{Haya S.~Aldosari
   \qquad  Catherine Greenhill \\
\small School of Mathematics and Statistics\\[-0.8ex]
\small UNSW Sydney\\[-0.8ex]
\small Sydney NSW 2052, Australia\\
\small \texttt{h.aldosari@student.unsw.edu.au} \quad \texttt{c.greenhill@unsw.edu.au}
}

\date{9 October 2020}  
\begin{document}
\maketitle

\begin{abstract}
An $r$-uniform hypergraph  $H$ consists of a set of vertices $V$ and a set of edges whose elements are $r$-subsets of $V$. We define  a hypertree to be a connected hypergraph which contains no cycles. A hypertree spans a hypergraph $H$ if it is a subhypergraph of $H$ which contains all vertices of $H$. Greenhill, Isaev, Kwan and McKay (2017) gave an asymptotic formula for the average number of spanning trees in graphs with given, sparse degree sequence. We prove an analogous result for $r$-uniform hypergraphs with given degree sequence $\kvec = (k_1,\ldots, k_n)$.  Our formula holds when $r^5 \kmax^3 = o((kr-k-r)n)$, where $k$ is the average degree and $\kmax$ is the maximum degree.
\end{abstract}
\vspace*{1\baselineskip}

\section{Introduction}

For $n \geq 3$, let 
$\kvec=\kvec(n) =(k_1, \ldots, k_n)$ be a sequence of non-negative integers.  A \emph{hypergraph} is a pair $(V,E)$ where  
$V$ is a set of vertices and $E$ is a multiset of 
multisubsets of $V$. The elements of $E$ are called \emph{edges}. Hence, under this definition we may have an edge containing a 
\emph{loop} if that edge has a vertex of multiplicity more than one. 
\purple{In this paper,} we focus on \emph{simple} hypergraphs: a hypergraph is simple if it has no loops and no repeated edges. 
For a positive integer $r$, we say a hypergraph is \emph{$r$-uniform} if every edge contains exactly $r$ vertices.
\purple{Some authors write ``hyperedge'' instead of edge when $r\geq 3$, but for simplicity we will continue to use ``edge''.}
All hypergraphs in this paper have vertex set $V=\{1,2, \ldots, n\}$.
The aim of this work is to estimate the average number of spanning hypertrees in $r$-uniform hypergraphs with a given degree sequence $\kvec$, when $r$ and the maximum degree are not too large. 
Applications of spanning hypertrees include the hypergraph analogue of the
Steiner tree problem studied by Warme~\cite{warme}.
%
%

We first define some terminology and notation. 

 Let $H=(V,E)$ be an $r$-uniform hypergraph. 
A \emph{loop}, or 1-cycle, is an edge which contains a repeated vertex.  
A 2-cycle is a hypergraph with two edges which intersect in at least~two vertices.
For any integer $\ell\geq 3$,
a \emph{cycle} with length $\ell$, or $\ell$-cycle, is a hypergraph with $\ell$ distinct edges which can be labelled as $e_1,\ldots, e_{\ell}$ such that 
there exists distinct vertices $v_1,\ldots, v_\ell$ with $v_i \in e_i \cap e_{i+1}$ for $i=1,\ldots, \ell$ (identifying $e_{\ell+1}$ with $e_1$).  
\purple{A hypergraph is \emph{linear} if each pair of edges overlaps in at most
one vertex. Equivalently, a hypergraph is linear if it contains no 2-cycles.}

A  (\emph{Berge}) \emph{path} in $H$ consists of a sequence $v_0,e_1,v_1,e_2,\ldots,e_{\ell},v_{\ell}$ where $v_0,v_1,\ldots, v_{\ell}$ are
distinct vertices, $e_1,\ldots, e_\ell$ are distinct edges and $ v_{i-1} ,v_i \in e_i$ for all $i=1,\ldots,\ell$. 
A hypergraph  is \emph{connected} if there is a path between every pair of vertices. We say that $H'=(V',E')$ is a \emph{subhypergraph} of $H=(V,E)$ if $V' \subseteq V$ and $E' \subseteq E$. 
 A \emph{hypertree} is a connected hypergraph which contains no cycles. Under this definition, any hypertree must be linear since  any pair of edges which intersect in at least two vertices generates a $2$-cycle. A \emph{spanning hypertree}  in $H$ is a subhypergraph of $H$ which forms 
hypertree containing all vertices of $H$. In other words, a hypergraph $T=(V(T),E(T))$ is a spanning hypertree in $H$ if $T$ is an acyclic, connected subhypergraph of $H$ with $V(T)=V(H)$. 
Clearly $T$ is $r$-uniform if $H$ is $r$-uniform.  We sometimes
abbreviate ``$r$-uniform hypertree'' to ``$r$-hypertree''.  
Note that an $r$-hypertree on $n$ vertices has exactly $\frac{n-1}{r-1}$ edges.

Let $\Hrk$ be the set of all simple $r$-uniform hypergraphs on $V$ with degree sequence $\kvec$. 
Denote by $\tau^{(r)}_{\kvec}$ the number of spanning $r$-hypertrees in a hypergraph $H$ chosen uniformly at random from $\Hrk$. If $\kvec$ is a regular degree sequence with $k_j=k$ for all $j\in [n]$ then we write $\tau^{(r)}_{n,k}$ for \purple{$\tau^{(r)}_{\kvec}$}. 
Define
\[
k=\dfrac{1}{n}\sum\limits_{j=1}^{n} k_i , \qquad \hat{k} =\bigg(\prod_{i=1}^{n} k_i\bigg)^{1/n}
\]
and write
\[ F^{(r)}(k,\hat{k})= \frac{(k-1)^{\nfrac{1}{2}} (r-1) }{n (kr-k-r)^{\nfrac{r+1}{2(r-1)}} } \left( \frac{ \hat{k} \,(r-1)^{k/r} (k-1)^{k-1} }{k^{\nfrac{kr-k}{r}} (kr-k-r)^{\nfrac{kr-k-r}{r(r-1) }}}\right)^n \ .
\]

\newpage
Our main result is stated below. 

\begin{theorem}\label{Espanninghypertree}
For $n\geq 3$, let $r=r(n) \geq 3$ be an integer number, and $\kvec=\kvec(n)=(k_1,\ldots,k_n)$ be a sequence of 
positive integers with maximum $\kmax$.   
Assume that $r$ divides $kn$ and $r-1$ divides $n-1$ for infinitely many values of $n$, and
perform asymptotics with respect to $n$ only along these values.
If $r^5 \, \kmax^3=o((kr-k-r)n)$ 
then the average number of spanning hypertrees in an $r$-uniform hypergraph with degree sequence $\kvec$ is
\begin{align*}
 & \E \, \tau^{(r)}_{\kvec} \\
&= F^{(r)}(k,\hat{k})\, 
\, \exp\left(\frac{kr-r-1}{2 (k-1) } 
- \frac{kr-r-2k+1}{2k(k-1)^2 n} \sum\limits_{i=1}^{n} (k_i -k)^2 
 +   O\left(\frac{r^5 \, \kmax^3}{(kr-k-r) n}\right) \right).
\end{align*}
\end{theorem}

 Noting that this theorem holds only when $r\geq 3$ and does not capture the correct asymptotic
expression in the case of graphs ($r=2$):  the factor $F^{(2)}(k,\hat{k})$ is correct but the
exponential factor is different.  
\purple{Simplicity for graphs is equivalent to conditioning on no 1-cycles and no 2-cycles. For $r$-uniform hypergraphs with $r\geq 3$, a simple hypergraph may contain 2-cycles when two edges overlap in more than one vertex. Hence the probability that a random hypergraph is simple takes a different form to the corresponding probability for graphs. }
Furthermore, the conclusion of Theorem~\ref{Espanninghypertree} also holds
if some entries of $\kvec$ equal zero, as then both $\tau_{\kvec}^{(r)}$ and $F^{(r)}(k,\hat{k})$ equal zero.

For $k$-regular $r$-uniform hypergraphs, we immediately obtain the following  corollary.

\begin{corollary}
For $n\geq 3$, let $r=r(n) \geq 3$ and $k = k(n)$ be positive integers.  Assume that $r$ divides $kn$ and
$r-1$ divides $n-1$ for infinitely many values of $n$, and perform asymptotics with respect to $n$ only along
these values.
If $r^5 \, k^3=o((kr-k-r)\, n)$
then the average number of spanning hypertrees in an $r$-uniform $k$-regular hypergraph is  
\begin{align*}
 \E \, \tau^{(r)}_{n,k} &= 
F^{(r)}(k,k)\, \exp\left( \frac{kr-r-1}{2(k-1)}
 +  O\left(\frac{r^5 \, k^3}{(kr-k-r) n} \right) \right).
\end{align*}
\end{corollary}

\subsection{Background}
The number of spanning trees in a graph $G$, also called the \emph{complexity} of $G$, is a very
well-studied parameter.
Greenhill et al.~\cite{GIKM17} gave an asymptotic formula for the average number of spanning trees in 
graphs with a given degree sequence, as long as the degree sequence is sufficiently sparse.  
This completed a sequence of papers beginning with McKay~\cite{McKay81}: see the history described in~\cite{GIKM17}.

There are several different definitions of hypertrees in the literature.
Our definition of hypertrees matches the definition given by Boonyasombat in~\cite{Bo84}. 
Siu~\cite{Siu02} gave a family of definitions of hypertrees, parameterised by the amount of
overlap allowed between edges.   Our definition of hypertrees matches what Siu calls 
``traditional hypertrees''~\cite[Section~1.2.1]{Siu02}: the other structures he studies
contain 2-cycles, as he allows edges to overlap in more than one vertex. 

Goodall and Mier~\cite{GoodMier11} investigated spanning trees in (non-random) 3-uniform hypergraphs, 
establishing some necessary conditions and some sufficient conditions for the existence
of a spanning tree.   They also proved that any Steiner triple system on $n$ vertices
has at least $\Omega((n/6)^{n/12})$ spanning trees~\cite[Theorem~4]{GoodMier11}.
A Steiner triple system can be viewed as a 3-uniform hypergraph such that every pair of distinct vertices
is contained in exactly one edge.
As far as we know, there is no prior work on the asymptotic number of spanning hypertrees in random 
uniform hypergraphs. 


We say that a sequence of $n$ positive integers $\xvec=(x_1,\ldots,x_n)$ is a \emph{suitable} degree sequence for  a hypertree in $ \Hrk$ if $x_i \leq k_i$ for all $i\in [n]$ and $\sum_{i=1}^n x_i = rt$
where $t= \frac{n-1}{r-1}$ is the number of edges in a hypertree on $[n]$. 
Denote by $\mathcal{T}$ the set of all  $r$-hypertrees on $n$ vertices, and for a suitable degree
sequence $\xvec$, define
\[\mathcal{T_{\xvec}}=\{T\in \mathcal{T} : T \mbox{ has degree sequence } \xvec  \}.
\]
Then
\[
|\mathcal{T}| = \frac{(n-1)! \, n^{t-1} }{t! \, (r-1)!^t}, 
\]
generalising Cayley's formula.
\purple{This result was given by Selivanov~\cite{selivanov}, see also~\cite{kolchin}.}
Alternative proofs using generalisations of Pr{\" u}fer codes were given in~\cite{lavault,shannigrahi,sivasubramanian2006spanning}.
A more general result was proved by Siu~\cite[Theorem~2.1]{Siu02} 
using a different definition of hypertrees, where edges are added
consecutively and a new edge may overlap a preceding edge in $d$ vertices.  Our definition of hypertree
corresponds to the case $d=1$.

For a suitable degree sequence $\xvec$, 
Bacher~\cite[Theorem 1.1]{Bacher11} proved that
 \begin{equation}\label{Tx}
 |\mathcal{T}_{\xvec} |=\frac{(r-1) (n-2)!}{(r-1)!^t \, \prod\limits_{i=1}^{n}(x_i -1)!}  .
 \end{equation}
This generalises a formula given by Moon~\cite{moon} in the case of graphs.


\subsection{Main ideas}\label{s:outline}

We write $(a)_b$  for the falling factorial $a(a-1) \cdots (a-b+1)$.  For any positive integer $a$, write $[a] = \{ 1,2,\ldots, a\}$.
Define $M(\kvec)$ as the sum of entries of $\kvec$, and $\Mtk = \sum_{i=1}^{n} (k_i)_2$. 
Suppose that $r$ divides $\Mk$ for infinitely many values of $n$ and take $n$ to infinity along these values. 
In~\cite{AldGre18}, we found an asymptotic formula for the probability that a random hypergraph from $\Hrk$ contains a given $r$-uniform hypergraph. 

\begin{theorem}\emph{\cite[Corollary 1.2]{AldGre18}}\,  \label{AldGre18}
For $n\geq 3$ and $r=r(n) \geq 3$, let $\kvec$ and $\kmax$ be defined as above. 
Let $X=X(n)$ be a given simple $r$-uniform hypergraph with degree sequence $\xvec$ and $t$ \purple{edges},
where $x_i \leq k_i$ for all $i=1,2,\ldots,n$. 
Define \[\beta =\frac{r^4 \kmax^3}{\Mkx} \, + \,\frac{t \, \kmax^3}{\Mkx^2} \, +\, 
  \frac{r\,t\, \kmax^4}{\Mkx^3}, \]
and assume that $\beta=o(1)$. Then 
the probability that a random hypergraph from $\Hrk$ contains every edge of $X$ is
\[
\frac{(\Mk/r)_{t} \, r!^{t}\, \prod_{i=1}^{n} (k_i)_{x_i} }{(\Mk)_{rt}}\,
  \exp\left(\frac{r-1}{2}\,\left(\frac{\Mtk}{\Mk}-  \frac{\Mtkx}{\Mkx} \right)+ O\left( \beta \right)\right). 
 \]
\end{theorem}

We follow the approach used by
Greenhill et al.~\cite{GIKM17} in the graph case.
For a given $r$-uniform hypertree $T$ on vertex set $[n]$, we can apply this result to find the 
probability that a random element of $\mathcal{H}_r(\kvec)$ contains $T$.
By summing over all hypertrees with a given degree sequence $\xvec$, we obtain the expected number
of spanning hypertrees with degree sequence $\xvec$ in a random element of $\mathcal{H}_r(\kvec)$.
Finally, by summing over all suitable degree sequences we complete the proof of Theorem~\ref{Espanninghypertree}.

Observe that the asymptotic formula given in Theorem~\ref{AldGre18} depends only on $r$, $\kvec$ and
$\xvec$ (up to the stated error term), and not on the specific edges of $X$.
In contrast, the corresponding formula of McKay~\cite[Theorem 4.6]{Mck85}
which was used by Greenhill et al.~\cite{GIKM17} in their enumeration of the average number of
spanning trees in graphs with given degrees, has terms which depend on the edges of $X$.  
This leads to differences in the calculation in the hypergraph case, as we do not have to average
over all trees with a given degree sequence as in~\cite{GIKM17}.  

Since all simple graphs are linear, it is possible that the asymptotic enumeration for the
expected number of spanning hypertrees in simple linear uniform hypergraphs will generalise the
formula for graphs.  This will be investigated in future work.


\section{The proof}\label{s:proof}

Recall that $k$ is the average of the elements of $\kvec$. 
Suppose that $\xvec = (x_1,\ldots, x_n)$ is a suitable degree sequence.
The next result follows by direct application of Theorem~\ref{AldGre18} (proof omitted),
using the fact that $M(\kvec) =kn$ and $\Mkx=kn-rt$.

 \begin{corollary}\label{hypertreeassubhypergraph}
Let $n\geq 3$,  $r=r(n) \geq 3$ be integers and $\kvec =\kvec(n)$ be a sequence of positive integers.  
Let $T$ be an $r$-hypertree with degree sequence $\xvec$ and $t=\nfrac{n-1}{r-1}$ edges, where $x_i \leq k_i$ for all $i\in [n]$. Define
  \begin{align*}
  &\lambda_0 = \frac{r-1}{2 kn}  \sum\limits_{i=1}^{n} (k_i)_2 ,\qquad \lambda(\xvec) = \frac{(r-1)^2}{2(kr-k-r)n +2r}  \sum\limits_{i=1}^{n} (k_i-x_i)_2 .
  \end{align*}
If $r^5 \, \kmax^3=o((kr-k-r)n)$ then
the probability that a random hypergraph from $\Hrk$ contains $T$ is
\[
\P(T) =\frac{(kn/r)_{t} \, r!^{t}\, \prod_{i=1}^{n} (k_i)_{x_i} }{(kn)_{rt}}\,
  \exp\left(\lambda_0 -  \lambda(\xvec)+ O\left( \frac{r^5 \kmax^3}{(kr-k-r)n} \right)\right). \\
 \]
\end{corollary}

 Define $\tau^{(r)}_{\kvec}(\xvec)$ as the number of $r$-hypertrees with degree sequence $\xvec$ in a hypergraph in $\Hrk$. Hence, using Corollary~\ref{hypertreeassubhypergraph} and linearity of expectation, we have 
\begin{align*}
 \E \,\tau^{(r)}_{\kvec} (\xvec)
&= \sum\limits_{T \in \mathcal{T}_{\xvec}}  \P (T) \notag \\
&=
 \frac{(kn/r)_{t} \, r!^{t}\, \prod_{i=1}^{n} (k_i)_{x_i} \, |\mathcal{T}_{\xvec}| }{(kn)_{rt}}\,
  \exp\left(\lambda_0 -  \lambda(\xvec)+  O\left( \frac{r^5 \kmax^3}{(kr-k-r)n} \right) \right) 
\end{align*}
since the formula from Corollary~\ref{hypertreeassubhypergraph} depends only on $\xvec$ and not
on the edges of $T$.   Applying (\ref{Tx}) gives
\begin{align*}
& \E \,\tau^{(r)}_{\kvec} (\xvec)\\
&=  \frac{(kn/r)_{t} \, r^t \,(r-1)\, (n-2)! }{(kn)_{rt}} 
 \left(\prod_{i=1}^{n}\dfrac{ (k_i)_{x_i}}{(x_i -1)!}\right)
\exp\left(\lambda_0 -  \lambda(\xvec)+  O\left( \frac{r^5 \kmax^3}{(kr-k-r)n} \right) \right).
\end{align*}
Now, we multiply and divide by $\binom{(k-1)n}{t-1}$ and rearrange, then sum over all possible 
suitable degree sequences $\xvec$, to obtain
\begin{align}\label{tauk}
\E \, \tau^{(r)}_{\kvec} &=   D^{(r)}_{\kvec} \,
\sum\limits_{\xvec}  \left\{
 \left(\prod_{i=1}^{n} \binom{k_i -1}{x_i -1}\right) / \binom{(k-1)n}{t-1} \exp\left( 
g(\xvec) + O\left( \frac{r^5 \kmax^3}{(kr-k-r)n} \right)
\right) \right\}
 ,
\end{align}
where 
\[
D^{(r)}_{\kvec} =  \frac{(kn/r)_{t} \, r^t \,(r-1)\,  (n-2)! \, \hat{k}^n }{(kn)_{rt}}  \quad \binom{(k-1)n}{t-1}
\, \quad \text{ and } \quad \,
g(\xvec)=\lambda_0 - \lambda(\xvec).
\]

Next, we work on $D^{(r)}_{\kvec}$.  By definition of $t$,
\begin{align*}
D^{(r)}_{\kvec} 
&=  \frac{r^t \,  \hat{k}^n \, n! \, (kn/r)!\,  ((k-1)n)!}{n \, (\nfrac{kn-rt}{r})! \, (kn)! \,  \quad t!}\\
&= \sqrt{\frac{k-1}{t(kn-rt)}} \,
\frac{\hat{k}^n \, n^{kn/r} \, (k-1)^{(k-1)n}}{k^{(kr-k)n/r}  \, (kn-rt)^{kn/r-t} \,\, t^t} \exp\left(O\left(\frac{r}{kn-rt}\right)\right)
\end{align*}
using Stirling's formula.
Now $kn-rt = t (kr-k-r) \, e^{O(k/t)}$ and $t^{-1} = \nfrac{r-1}{n}\, e^{O(1/n)}$, so
\begin{align}
D^{(r)}_{\kvec} 
&= F^{(r)}(k,\hat{k})\,\, \exp\left( O\left(\frac{r+k}{kn-rt}\right)\right).
\label{Akr}
\end{align}
Hence, since the error term in (\ref{Akr}) is dominated by the error term in (\ref{tauk}),
\begin{align*}
\E \, \tau^{(r)}_{\kvec} &=   F^{(r)} (k,\hat{k}) \,
\sum\limits_{\xvec}  \left\{
 \left(\prod_{i=1}^{n} \binom{k_i -1}{x_i -1}\right) / \binom{(k-1)n}{t-1} \exp\left( 
g(\xvec) + O\left( \frac{r^5 \kmax^3}{(kr-k-r)n} \right)
\right) \right\}.
\end{align*}

we only need to find the sum over $\xvec$ in this expression in order to prove the main result. 
This sum can be estimated using a similar approach to~\cite{GIKM17}.

 First, a slight generalisation of~\cite[Lemma 5.1]{GIKM17} is stated below, in our notation.

\begin{lemma}\emph{\cite[Lemma 5.1]{GIKM17}} \label{tauk2}
Partition $[(k-1)n]$ into $n$ sets $A_1, \ldots, A_{n}$, where $|A_i| = k_i -1$ for $i=1,\ldots, n$. 
Let $C$ be a subset of $[(k-1)n]$ of size $t-1$, chosen uniformly at random. 
Define a random vector $\X=\X(C)=(X_1, \cdots, X_n)$ by $X_j=|A_j \cap C| +1$. Then
\begin{align*}
\E \tau^{(r)}_{\kvec} =   F^{(r)}(k,\hat{k}) \,\, \E\exp\left(g(\X) +O\left( \frac{r^5 \kmax^3}{(kr-k-r)n} \right) \right).
\end{align*}
\end{lemma}

The expectation of $\exp({g(\X)})$ can be easily determined by computing $e^{\E g(\X)}$ after proving that $\E(e^{g(\X)}) \sim e^{\E g(\X)}$. This can be done with the assistance of \cite[Corollary 2.2]{GIKM17}, restated below. \purple{Note that this is not an asymptotic result but gives an explicit bound for given values of $N$, $s$ and a given function $h$.}

\begin{lemma}\emph{\cite[Corollary 2.2]{GIKM17}}\label{Eetofn}
Let $\binom{[N]}{s}$ be the set of $s$-subsets of $\{1,\ldots, N\}$ and let $h: \binom{[N]}{s} \rightarrow \mathbb{R}$ be given. Let $C$ be a uniformly random element of $\binom{[N]}{s}$. Suppose that, for any $A, A^' \in \binom{[N]}{s}$ with $s-1$ elements in common,  there exists $\alpha \geq 0$ such that
\[
| h(A)-h(A^')| \leq \alpha.
\]
Then 
\begin{align}
\E \exp( h(C)) = \exp \left( \E h(C) + K\right),
\end{align}
where $K$ is a real constant such that $0 \leq K \leq \nfrac{1}{8} \,\min\{s, N-s \} \alpha^2$. 
Furthermore, for any real $z >0$,
\begin{align*}
\Pr(|h(C) - \purple{\E h(C)}| \geq z) \leq \exp \left(\frac{- 2 z^2}{\purple{\min\{s,N-s\}} \alpha^2}\right)
\end{align*}
\end{lemma}

 Two suitable degree sequences $\xvec, \xvec^'$ are \emph{adjacent} if they are different in two entries $i,j$ such that $x_i' = x_i +1, \, x_j' = x_j -1$.  These sequences, respectively, correspond to two sets $A,A' \in \binom{[(k-1)n]}{t-1}$ with $t-2$ vertices in common.

\begin{lemma}\label{Eegx}
  \[
\E\, e^{g(\X)} = \exp\left(\E\, g(\X) +O\left(\frac{ r^3\, \kmax^2}{(kr-k-r)^2\, n}\right) \right).
\]
\end{lemma}
\begin{proof}
For adjacent suitable degree sequences $\xvec, \xvec'$  and from definition of $g(\xvec)$ we have
\begin{align*}
|g(\xvec) - g(\xvec')| &= |\lambda(\xvec) -\lambda(\xvec')| \\
&=\frac{(r-1)^2|2(k_i-x_i-1) - (k_j-x_j)|}{2(kr-k-r)n+2r} \\
&=O\left(\frac{r^2\, \kmax }{(kr-k-r)n}\right).
\end{align*}
Therefore we can apply Lemma~\ref{Eetofn} where
\[h(C) =g(\X(C)), \quad N=(k-1)\, n , \quad  s =t-1 \mbox { and } \alpha=O\Big(\frac{r^2 \, \kmax}{(kr-k-r)n}\Big).\]
Since $ N-s =  (kr-k-r)t + 1 >  s$ \purple{and $rt=O(n)$}, we have
\[
K = O\left(\frac{t\, r^4\kmax^2}{(kr-k-r)^2\, n^2}\right)=O\left(\frac{ r^3\, \kmax^2}{(kr-k-r)^2\, n}\right).
\]
This completes the proof.  
\end{proof}

The distribution of \purple{the vector $\X$ from} Lemma~\ref{tauk2} is called a multivariate hypergeometric distribution with parameters $(t-1,\kvec)$ as defined in~\cite[equation (39.1)]{JKB97}. Therefore, for non-negative integers $a,b$ and using \cite[equation (39.6)]{JKB97} we can compute the expectation of $(X_j -1)_a$ as
\begin{align}\label{EX-1}
\E \,( (X_j -1)_a) = \frac{(t-1)_{a}}{((k-1)n)_{a}} \,(k_j -1)_a.
\end{align}

We use this expression to estimate $\E \, (g(\X))$ as follow.
\begin{lemma}\label{Egx}
\begin{align*}
 \E \, g(\X) = \frac{kr -r-1 }{2 (k-1) } -  
\frac{kr-r-2k+1}{2k(k-1)^2 n} \sum\limits_{i=1}^{n} (k_i -k)^2 
  +O\left(\frac{r \, \kmax}{kn}\right).  
\end{align*}
\end{lemma}

\begin{proof}
Recall $g(\xvec)=\lambda_0 -\lambda(\xvec)$, where $\lambda_0$ and $\lambda(\xvec)$ are
defined in Corollary~\ref{hypertreeassubhypergraph}. 
We restate  $\lambda(\X)$ as
\begin{align}
\lambda(\X) 
&=  \frac{r-1}{2(kn-rt)} \sum\limits_{i=1}^{n} \Big(  (k_i -1)_2 -2 (k_i -2)  (X_i -1) +  (X_i -1)_2 \Big) .
\label{lambdaX}
\end{align}
Applying (\ref{EX-1}) on the expected value of the summand in (\ref{lambdaX}) implies
\begin{align}\label{lambda}  
(k_i -1)_2 -2 (k_i -1)_2 \frac{t-1}{(k-1)n} +(k_i -1)_2 \frac{(t-1)_2}{((k-1)n)_2}.
\end{align}
Taking out a common factor and using the identity $t-1=rt-n$,  (\ref{lambda}) can be rewritten as
\begin{align*}
&\frac{(k_i -1)_2}{((k-1)n)_2}\Big( ((k-1)n)_2 -2 (rt-n)  (kn-n-1) +(rt-n)_2 \Big)\\
&=  \frac{ (kn-rt) (kn-rt-1)\, (k_i -1)_2}{((k-1)n)_2}.
\end{align*}
Substuting this into (\ref{lambdaX}), the expected value of $\lambda (\X)$ is 
\begin{align*}
\frac{(r-1)(kn-rt-1)}{2 ((k-1)n)_2 } \sum_{i=1}^{n} (k_i -1)_2 
=  \frac{ kr-k-r}{2 (k-1)^2 n} \sum\limits_{i=1}^{n} (k_i -1)_2 +O\left( \frac{r \kmax}{kn}\right).
\end{align*}
As a result, the expectation of $g(\X)$ is 
\begin{align}\label{Egx1}
\E g(\X) = \frac{r-1}{2kn} \sum\limits_{i=1}^{n} (k_i )_2 -    \frac{ kr-k-r}{2 (k-1)^2 n} \sum\limits_{i=1}^{n} (k_i -1)_2 +O\left( \frac{r \kmax}{kn}\right).
\end{align}

The first sum in this equation can be written as
$ k (k-1)n +\sum_{i=1}^{n} (k_i -k)^2$,
while  the second sum is
 $
 (k-1) (k-2) n+ \sum_{i=1}^{n} (k_i -k)^2$.
Substituting these  into (\ref{Egx1}) and \purple{simplifying} the result will complete the proof of this lemma.
\end{proof}
 {\bf Proof of Theorem~\ref{Espanninghypertree}:}
 
 \begin{proof}
Substitution  from Lemma~\ref{Eegx} and Lemma~\ref{Egx} into the expression of Lemma~\ref{tauk2} proves the required result, with combined error term
\begin{align*}
O\left(\frac{r^5 \kmax^3}{(kr-k-r)n}+ \frac{r^3 \kmax^2}{(kr-k-r)^2 n} + \frac{r \kmax}{kn} \right)
=O\left( \frac{r^5 \kmax^3}{(kr-k-r)n} \right).
\end{align*}
\end{proof}

\end{document}